\documentclass[12pt]{amsart}

\usepackage[latin1]{inputenc}
\usepackage{amsmath}
\usepackage{amsfonts}
\usepackage{amssymb}
\usepackage{graphics}
\usepackage{enumerate}
\usepackage{subfigure}
\usepackage{amssymb,amsmath,amsthm,amscd,epsf,latexsym,verbatim,graphicx,amsfonts}
\input epsf.tex

\usepackage{amscd}
\usepackage{amsmath}
\usepackage{amssymb}
\usepackage{amsthm}
\usepackage{epsf}
\usepackage{latexsym}
\usepackage{verbatim}
\input epsf.tex

\newtheorem{theorem}{Theorem}[section]
\newtheorem{lemma}[theorem]{Lemma}
\newtheorem{corollary}[theorem]{Corollary}

\numberwithin{equation}{section}

\newcommand{\bbR}{\mathbb{R}}
\newcommand{\bbC}{\mathbb{C}}

\newcommand{\bbD}{\mathbb{D}}
\newcommand{\bbN}{\mathbb{N}}
\newcommand{\eitheta}{e^{i\theta}}

\newcommand{\mcc}{\mathcal{C}}
\newcommand{\mccn}{\mathcal{C}^{(n)}}
\newcommand{\mccm}{\mathcal{C}^{(M)}}

\newcommand{\mcln}{\mathcal{L}^{(n)}}

\newcommand{\mcmn}{\mathcal{M}^{(n)}}
\newcommand{\Real}{\textrm{Re}}

\newcommand{\thalph}{\theta_{\alpha}}
\newcommand{\thalphm}{\theta_{\alpha_M}}
\newcommand{\eithalph}{e^{i\theta_{\alpha}}}

\newcommand{\upp}{(n)}
\newcommand{\gap}{(-\thalph,\thalph)}
\newcommand{\nri}{n\rightarrow\infty}

\begin{document}

\title[ ] {Zeros of non-Baxter Paraorthogonal Polynomials on the Unit Circle}

\bibliographystyle{plain}

\thanks{  }

\maketitle

\begin{center}
\textbf{Brian Simanek\footnote{This research was partially supported by an NSF GRFP grant.}}\\
\textit{\small{California Institute of Technology\\
MC 253-37, 1200 E. California Blvd. Pasadena, CA 91125.}}\\
\small{bsimanek@caltech.edu}
\end{center}

\begin{abstract}
We provide leading order asymptotics for the size of the gap in the zeros around $1$ of paraothogonal polynomials on the unit circle whose Verblunsky coefficients satisfy a slow decay condition and are inside the interval $(-1,0)$.  We also include related results that impose less restrictive conditions on the Verblunsky coefficients.
\end{abstract}

\vspace{4mm}


\vspace{2mm}


\vspace{2mm}

\normalsize

\section{Introduction}\label{intro}

In this paper, we will let $\mu$ be a probability measure on the unit circle $\partial\bbD$ with infinite support.  In the Hilbert Space $L^2(\partial\bbD,d\mu)$, the set $\{1,z,z^2,z^3,\ldots\}$ is a linearly independent set so by applying Gram-Schmidt orthogonalization, we can obtain the orthonormal polynomials $\{\varphi_n(z)\}_{n\geq0}$ where $\varphi_n(z)$ is a polynomial of degree exactly $n$ with positive leading coefficient and we have the relation
\[
\langle\varphi_n,\varphi_m\rangle:= \int_{\partial\bbD}\overline{\varphi_n(z)}\varphi_m(z)d\mu(z)=\delta_{mn}.
\]

If we consider instead the monic orthogonal polynomials $\{\Phi_n(z)\}_{n\geq0}$, then it is well known (see Section 1.5 in \cite{OPUC1}) that these polynomials satisfy the recursion relation
\[
\Phi_{n+1}(z)=z\Phi_n(z)-\overline{\alpha}_n\Phi_n^*(z)
\]
where
\[
\alpha_n\in\bbD\quad,\quad\Phi_n^*(z)=z^n\overline{\Phi_n(1/\bar{z})}.
\]
The coefficients $\{\alpha_n\}_{n\geq0}$ are called the \textit{Verblunsky coefficients} and the recursion is called the \textit{Szeg\H{o} recursion} (see \cite{OPUC1} for further information).

It is often both interesting and insightful to study the relationship between properties of the measure and properties of the corresponding sequence $\{\alpha_n\}_{n\geq0}$.  For example, if we write
\begin{eqnarray*}\label{decomp}
d\mu(\eitheta)=w(\theta)\frac{d\theta}{2\pi}+d\mu_s(\eitheta)
\end{eqnarray*}
with $\mu_s$ singular with respect to Lebesgue measure,
and $\alpha_n\equiv\alpha\in\{z:|z+\frac{1}{2}|<\frac{1}{2}\}$ then $d\mu_s=0$ and $w(\theta)=0$ for $\theta\in[-2\arcsin|\alpha|,2\arcsin|\alpha|]$ (see Theorem 1.6.13 in \cite{OPUC1}).  We will refer to this result again later.

One useful tool for studying the orthogonal polynomials of a measure $\mu$ is the CMV matrix, which is a semi-infinite unitary matrix whose spectral measure is $\mu$ and is such that the characteristic polynomial of the upper-left $n\times n$ block is exactly $\Phi_n(z)$ (see Chapter 4 in \cite{OPUC1}).  However, since the upper-left $n\times n$ block is not unitary, one often introduces the so-called \textit{paraorthogonal polynomials} $\{\Phi_n^{(\beta_n)}(z)\}_{n\geq1}$ defined by
\begin{eqnarray*}
\Phi_{n+1}^{(\beta_n)}(z)=z\Phi_n(z)-\overline{\beta}_n\Phi_n^*(z)
\end{eqnarray*}
where $\beta_n\in\partial\bbD$.  These polynomials and their zeros have been studied extensively in recent years (see for example \cite{KStoi,FineIV,FineI,RankOne,Stoiciu,Lilian} and references therein).

It is well known (see \cite{Lilian}) that all of the zeros of $\Phi_n^{(\beta_n)}(z)$ are simple and lie on the unit circle so we can denote them by $\{\zeta_1^{(n)},\ldots,\zeta_n^{(n)}\}$ ordered counterclockwise starting from $1$.  Our investigation is motivated by a picture appearing in Chapter 8 in \cite{OPUC1} and also in \cite{Beware}.  It is also closely related to Conjecture D in Section 4 of \cite{FineI}.  In \cite{FineI}, the author points out that if the Verblunsky coefficients have a slow decay property (we will make this more precise later), then the zeros of the orthogonal polynomials appear to exhibit clock spacing on the unit circle away from $\theta=0$, while there is a large gap around $\theta=0$.
\begin{figure}[h!]\label{zeros}
\begin{center}
\includegraphics[scale=.5]{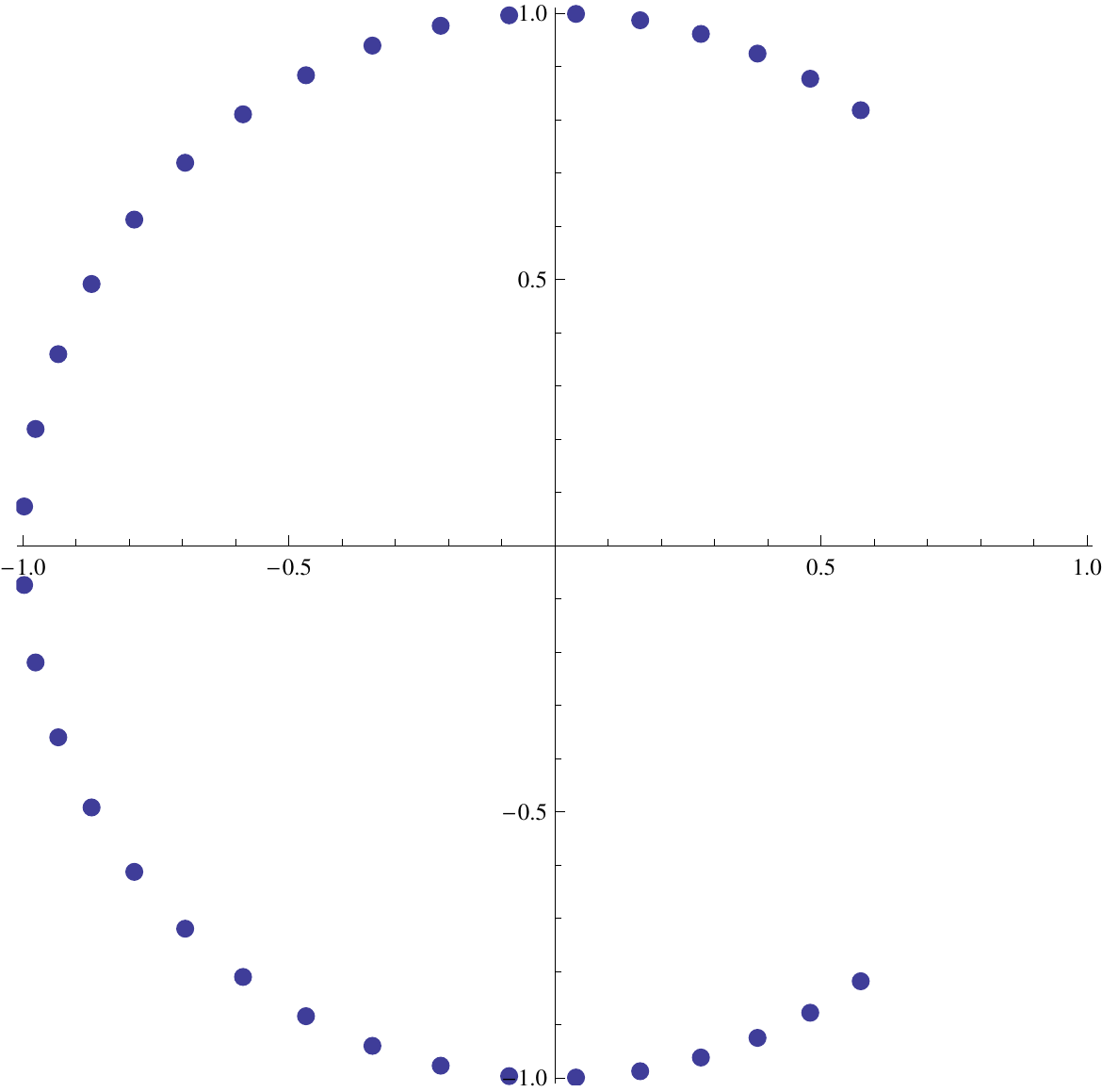}
\caption{A Mathematica plot of the zeros of $\Phi_{34}^{(-1)}(z)$ when $\alpha_n=-(n+2)^{-1/4}$.}
\end{center}
\end{figure}
The main result of this note is to give leading order asymptotics for the size of the gap in the zeros of the paraorthogonal polynomials around $\theta=0$ when $\beta_n\in(\pi/2+\epsilon,3\pi/2-\epsilon)$ for some $\epsilon>0$ and all $n$.  Along the way we will develop an idea suggested in \cite{Beware}.  In that paper, the authors deal with an explicit measure instead of explicit Verblunsky coefficients but observe the same phenomenon of a gap in the zeros of the orthogonal polynomials.  The authors then provide numerical evidence to suggest that the Verblunsky coefficients do satisfy a slow decay condition and attribute the gap in the zeros to the suggestion that the polynomial $\Phi_n$ ``thinks" it belongs to a measure with constant Verblunsky coefficients because the decay is so slow.
Our methods will allow us to make this idea more precise and put a condition on the sequence $\{\alpha_n\}_{n\geq0}$ that causes the zeros to exhibit a gap property as in the constant Verblunsky coefficient case.

We will often be interested in situations where the Verblunsky coefficients decay monotonically (in absolute value) to $0$.  However, this condition is more restrictive than is necessary for our results.  We will say that a sequence of Verblunsky coefficients $\{\alpha_n\}_{n\geq0}$ has \textit{slow decay controlled by $f$} if there is a non-positive function $f(n)$ so that
\begin{enumerate}
\item\label{either} either $f(n)=-Cn^{-b}$ for some $b\in(0,1)$ and $C>0$ or $\lim_{\nri}\sqrt{n}\,f^2(n-\sqrt{n})=\infty$,
\item $f(n)<f(n+1)<0$ for all $n>0$ and $\lim_{\nri}f(n)=0$,
\item $\lim_{\nri}\frac{f(n-k\sqrt{n})}{f(n)}=1$ for all $k\in\bbN$,
\item $\lim_{\nri}\frac{\alpha_n}{f(n)}=1$,
\item $\Real[\alpha_n]\leq f(m)$ when $n\leq m$ and $m$ is sufficiently large.
\end{enumerate}
Of course the two cases in condition (i) are not disjoint, but this will only create some flexibility in how we proceed in proofs.  Following the terminology in Section 5 of \cite{FineI}, orthogonal polynomials with $\ell^1$ Verblunsky coefficients are said to be in the \textit{Baxter Class}, so we are interested only in the non-Baxter case (results concerning Baxter class Verblunsky coefficients can be found in \cite{FineI}).
Our main result is the following

\begin{theorem}\label{kthzero}
Let $k\in\bbN$, $\tilde{\epsilon}>0$, and $\xi\in(\frac{\pi}{2}+\tilde{\epsilon},\frac{3\pi}{2}-\tilde{\epsilon})$ be fixed and let $\beta=e^{i\xi}$.  Let $\{\alpha_n\}_{n\geq0}$ be a sequence of real Verblunsky coefficients having slow decay controlled by $f$.
For each $\delta>0$, there exists $N_k=N_k(\delta)\in\bbN$ so that if $M>N_k$ then $\zeta_k^{(M)}$ obeys
\[
\left|\frac{\arg\left(\zeta_k^{(M)}\right)}{2|f(M)|}-1\right|<\delta.
\]
Similarly, we have
\[
\left|\frac{2\pi-\arg\left(\zeta_{M-k+1}^{(M)}\right)}{2|f(M)|}-1\right|<\delta.
\]
\end{theorem}
In \cite{Beware}, the authors provide numerical and strong heuristic evidence for Theorem \ref{kthzero} in the context of a closely related problem.  In Section \ref{extensions} we will prove a generalization of Theorem \ref{kthzero} that imposes less restrictive conditions on the Verblunsky coefficients.  As a special case of Theorem \ref{kthzero}, we get the following corollary -  conjectured in \cite{FineI} - which is of special interest.

\begin{corollary}\label{specialcase}
If $\alpha_n=-(n+2)^{-b}$ with $0<b<1$ then the size of the gap in the zeros of $\Phi_n^{(-1)}(z)$ around $z=1$ is $4\arcsin|\alpha_n|$ to leading order.  The neighboring zero spacings are of smaller order.
\end{corollary}

\noindent The latter possibility in condition (i) on $f$ means we can also apply Theorem \ref{kthzero} to the sequence of Verblunsky coefficients given for example by $\alpha_n=\frac{-1}{\log(n+3)}$.

Following the wisdom that an equality is two inequalities, we will prove Theorem \ref{kthzero} in two steps; first we prove a lower bound on the size of the gap and then an upper bound.  We will prove the lower bound by using the Szeg\H{o} recursion to control the phases of the Blaschke product $\Phi_n(\eitheta)/\Phi_n^*(\eitheta)$ when $\theta$ is sufficiently close to $0$.  To prove the upper bound, we will find a large collection of mutually orthogonal approximate eigenvectors for a unitary matrix whose characteristic polynomial is $\Phi_n^{(-1)}(z)$ and apply the variational principle.  The result for $\beta\neq-1$ will follow from the interlacing of zeros for distinct values of $\beta$ (see Theorem 1.3 in \cite{RankOne}).

In the next section we will prove the key lemma for the lower bound and explain how it relates to the ideas put forward in \cite{Beware}.  In Section \ref{sizegap}, we will prove Theorem \ref{kthzero}.  In Section \ref{extensions} we will discuss modifications that we can make to the hypotheses of Theorem \ref{kthzero} and state a related theorem for the orthogonal polynomials on the unit circle that can be proved using our techniques.  Along the way we will add our result to several previously known results to explain several features of Figure \ref{zeros}.

\vspace{2mm}

\noindent\textbf{Acknowledgements.}  It is a pleasure to thank Barry Simon for introducing me to this problem and for much useful discussion.

\section{The Key Lemma}\label{keylemma}

We recall a version of the Szeg\H{o} recursion used by Khruschev and given in \cite{OPUC2} as equation (9.2.16), which states
\begin{eqnarray}\label{kru}
b_{n+1}(\eitheta)=\frac{\eitheta b_n(\eitheta)-\bar{\alpha}_n}{1-\alpha_n\eitheta b_n(\eitheta)}
\end{eqnarray}
where $b_m(\eitheta)=\varphi_m(\eitheta)/\varphi_m^*(\eitheta)$ (notice that $|b_m(\eitheta)|=1$).  Although Theorem \ref{kthzero} requires the Verblunsky coefficients to be real, we will prove our lower bound under more general hypotheses (see also Corollary \ref{nopp} below).

In order to state the key lemma, we first define the region $P_\alpha$ for $\alpha\in\{z:|z+\frac{1}{2}|<\frac{1}{2}\}$ as all $z\in\bbD$ with real part less than or equal to $\Real(\alpha)$.  Also, following the notation from \cite{OPUC1}, we define $\theta_{\alpha}=2\arcsin|\alpha|$.
The following lemma is the key to the lower bound.

\begin{lemma}\label{openchoice}
Let $\alpha$ be a real number in $(-1/2,0)$.  If $\alpha_0,\alpha_1,\ldots,\alpha_{n-1}\in P_{\alpha}$ then
\[
\arg(b_n(\eitheta))\in\left(-\frac{\pi}{2}+\arcsin|\alpha|,\frac{\pi}{2}-\arcsin|\alpha|\right)
\]
for all $\theta\in\gap$.
\end{lemma}

\begin{proof}
The proof is by induction.  By equation 10.8 in \cite{FineIV}, the phase of the Blaschke product $b_n(\eitheta)$ increases with $\theta$.  Furthermore, if we prove that $\arg(b_n(e^{i\thalph}))<\frac{\pi}{2}-\arcsin|\alpha|$ then we will have also proven that $\arg(b_n(e^{-i\thalph}))>-\frac{\pi}{2}+\arcsin|\alpha|$ by symmetry. Then using the monotonicity of $\arg(b_n(\eitheta))$ the lemma will follow.

Using the recursion mentioned above, we would like to choose $\alpha_0$ so that
\[
\arg(\eithalph-\bar{\alpha}_0)-\arg(1-\alpha_0\eithalph)<\pi/2-\arcsin|\alpha|
\]
where we have the $\arg$ function take values in $[-\pi,\pi)$.  With this in mind, choose any $\alpha_0$ in $P_\alpha$.
By construction, $\arg(\eithalph-\bar{\alpha}_0)-\arg(1-\alpha_0\eithalph)$ is maximized over $\alpha_0\in P_{\alpha}$ when $\alpha_0=e^{i(\frac{\pi}{2}+\arcsin|\alpha|)}$.
Therefore,
\begin{eqnarray*}
\arg(\eithalph-\bar{\alpha}_0)-\arg(1-\alpha_0\eithalph)&<&
\arg(\eithalph-e^{-i(\frac{\pi}{2}+\arcsin|\alpha|)})-\arg(1-e^{i(\frac{\pi}{2}+\arcsin|\alpha|)}\eithalph)\\
&\leq&\frac{\pi}{2}-\arcsin|\alpha|
\end{eqnarray*}
where we used only elementary geometry to derive this last upper bound.
Therefore, we have completed the base case by showing that choosing $\alpha_0\in P_{\alpha}$ leads to
\[
\arg(b_1(\eithalph))<\frac{\pi}{2}-\arcsin|\alpha|.
\]

For our induction hypothesis, let us assume that $\alpha_0,\alpha_1,\ldots,\alpha_{j-1}$ have been chosen from $P_\alpha$ so that $b_j(z)$ maps the circular arc $(-\thalph,\thalph)$ to some sub-arc of
$(-\frac{\pi}{2}+\arcsin|\alpha|,\frac{\pi}{2}-\arcsin|\alpha|)$.
We will show that choosing $\alpha_j$ from $P_{\alpha}$ implies the same condition holds for $b_{j+1}(z)$.  Again by symmetry, it suffices to put an upper bound on $\arg(b_{j+1}(\eithalph))$.  Our induction hypothesis implies that
\[
-\pi/2+3\arcsin|\alpha|<\arg(\eithalph b_j(\eithalph))<\pi/2+\arcsin|\alpha|
\]
so that we have a picture similar to the one we used when choosing $\alpha_0$.  Therefore, the argument of $b_{j+1}(\eithalph)$ is again maximized by setting $\alpha_j=e^{i(\frac{\pi}{2}+\arcsin|\alpha|)}$.  Observe that
for $s\in\bbD\cup\{e^{ix}:x\in(\frac{\pi+\thalph}{2},\frac{3\pi}{2})\}$ and $t\in(-\frac{\pi}{2},\frac{\pi+\thalph}{2})$, we have
\[
\frac{d}{dt}\bigg(\arg(e^{it}-s)-\arg(1-se^{it})\bigg)>0
\]
(see Equation 6.10 in \cite{Stoiciu}).
This implies
\begin{eqnarray*}
\arg(b_{j+1}(\eithalph))&\leq&
\arg(\eithalph b_j(\eithalph)-e^{-i(\frac{\pi}{2}+\arcsin|\alpha|)})
-\arg(1-e^{i(\frac{\pi}{2}+\arcsin|\alpha|)}\eithalph b_j(\eithalph))\\
&<&\arg(e^{i(\pi/2+\arcsin|\alpha|)}-e^{-i(\frac{\pi}{2}+\arcsin|\alpha|)})
-\arg(1-e^{i(\frac{\pi}{2}+\arcsin|\alpha|)} e^{i(\pi/2+\arcsin|\alpha|)})\\
&=&\frac{\pi}{2}-\arg(1+\eithalph)\\
&=&\frac{\pi}{2}-\arcsin|\alpha|.
\end{eqnarray*}
This completes the inductive step.
\end{proof}

Let $\alpha<0$ and let $\Phi_n(z)$ and $\Psi_n(z)$ be the monic orthogonal and second kind polynomials respectively corresponding to the sequence of Verblunsky coefficients $\{\alpha_n\}_{n\geq0}$.  Define the measure $\mu_{-n}=\mu(\alpha_0,\ldots,\alpha_{n-1},\alpha,\alpha,\ldots)$ with corresponding Caratheodory function $F_{-n}$ (see Chapter 1.3 in \cite{OPUC1}) and let $\mu_0=\mu(\alpha,\alpha,\alpha,\ldots)$ with corresponding Caratheodory function $F_0$.  Let us recall a formula of Peherstorfer given as Theorem 3.4.2 in \cite{OPUC1}.  Using the notation above, it tells us that
\[
F_{-n}(z)=\frac{\Psi^*_n(z)-\Psi_n(z)+(\Psi^*_n(z)+\Psi_n(z))F_0(z)}
{\Phi^*_n(z)+\Phi_n(z)+(\Phi^*_n(z)-\Phi_n(z))F_0(z)}.
\]
By Lemma 3.2.15 in \cite{OPUC1}, the measure $\mu_{-n}$ has a pure point if and only if the Caratheodory function $F_{-n}$ has a point at which it blows up in the arc $(-\thalph,\thalph)$.  This occurs when
\begin{eqnarray*}
F_0(\eitheta)&=&\frac{\Phi^*_n(\eitheta)+\Phi_n(\eitheta)}{\Phi_n(\eitheta)-\Phi^*_n(\eitheta)}
=\frac{b_n(\eitheta)+1}{b_n(\eitheta)-1}
=-i\cot\bigg(\frac{\arg(b_n(\eitheta))}{2}\bigg).
\end{eqnarray*}
One can calculate that for $\alpha\in\bbR$,
\[
-iF_0(\eitheta)=\frac{\alpha\cot(\theta/2)+\csc(\theta/2)\sqrt{\alpha^2-\sin^2(\theta/2)}}{1+\alpha}
\]
when $\theta$ is in $\gap$.  Using Lemma \ref{openchoice}, one can control the argument of $b_n(\eitheta)$ so that $iF_0(\eitheta)\neq\cot(\arg(b_n(\theta))/2)$ for all $\theta\in\gap$ and hence arrive at the following

\begin{corollary}\label{nopp}
Let $\alpha$ be a real number in $(-1/2,0)$.  If $\alpha_0,\alpha_1,\ldots,\alpha_{n-1}\in P_{\alpha}$ then the resulting measure $\mu(\alpha_0,\alpha_1,\ldots,\alpha_{n-1},\alpha,\alpha,\ldots)$ has no pure points in $\gap$.
\end{corollary}

\noindent\textit{Remark.}  Corollary \ref{nopp} refines the estimate given by Theorem 3.4.7 in \cite{OPUC1}, which puts a bound on the number of pure points (under more general hypotheses).  A related result can also be found in \cite{Nevai} for the special case when all the $\alpha_j\in\bbR$ where it is shown that $1$ is not a pure point.

\vspace{1mm}

Corollary \ref{nopp} makes precise the claim in \cite{Beware} that the gap in the zeros results from the polynomials ``thinking" they are in the constant Verblunsky coefficient case.  Put differently, the orthogonal polynomial $\Phi_n$ is the monic degree $n$ orthogonal polynomial for \textit{some} measure having exactly the same support as the spectral measure corresponding to constant Verblunsky coefficients (which we described in Section \ref{intro}).  The zeroes of $\Phi_n$ therefore lie in the convex hull of this measure (this is by Fejer's Theorem, see Theorem 1.7.19 in \cite{OPUC1}).  However, since $\eitheta\in\overline{P}_{\alpha}$ for all $\theta\in(\frac{\pi}{2}+\arcsin|\alpha|,\frac{3\pi}{2}-\arcsin|\alpha|)$, we can move $\alpha_{n-1}$ to such a point on $\partial\bbD$ inside $P_{\alpha}$ and retain this restriction on the location of the zeros of $\Phi_n$.  Therefore, by the Hurwitz Theorem we have a lower bound on the size of the gap for all $\beta$ of the form considered in Theorem \ref{kthzero} when $\alpha$ is sufficiently small (i.e. $M$ is sufficiently large).  In the next section we will complete the proof of the main theorem.

\section{Finding The Zeros}\label{sizegap}

To establish an upper bound on the distance from $\zeta_k^{(M)}$ to $1$, we resort to a tool from spectral theory.
It is well-known that for any normal matrix $N$, if there exists a unit vector $\nu$ such that
\[
||(N-z_0)\nu||<\epsilon
\]
then $N$ has an eigenvalue in the ball $\{z:|z-z_0|<\epsilon\}$ (see Theorem 5.9 in \cite{Hislop} for a proof and see \cite{DaviesSimon} and Theorem 4.1 in \cite{Bandtlow} for generalizations).  The paraorthogonal polynomials $\Phi_{n}^{(\beta)}$ are obtained by setting $\alpha_{n-1}=\beta\in\partial\bbD$.  The resulting CMV matrix then decouples and the upper-left $n\times n$ block $\mccn_\beta$ is in fact a unitary (and hence normal) matrix.  Since
\[
\Phi_{n}^{(\beta)}(z)=\det(z-\mccn_{\beta}),
\]
then in order to show the existence of a zero of $\Phi_{n}^{(\beta)}$ within a certain proximity of $1$, it will suffice to find a unit vector $\nu_n$ such that $||(\mccn_\beta-1)\nu_n||$ is small.  We will define $\nu_n$ for each $n$ by
\[
(\nu_{n})_j=
\begin{cases}
(j-\gamma_{n})(n-j), &\mbox{ $j\in(\gamma_{n},n)$ even}\\
i(j-\gamma_{n})(n-j), &\mbox{ $j\in(\gamma_{n},n)$ odd}\\
0, &\mbox{ otherwise}
\end{cases}
\]
where $\gamma_n$ will depend on $f$.

If $\alpha_m=-Cm^{-b}+o(m^{-b})$ with $b\in(0,1)$ then we set $\gamma_n=n-n^{\frac{1+b}{2}}$ (more precisely the closest integer to this quantity but we continue to denote it by $n-n^{\frac{1+b}{2}}$).  Using the form of the CMV matrix in Section 4.2 in \cite{OPUC1} we can calculate
$\|(\mccn_{-1}-1)\nu_{n}\|^2=\sum_{j=1}^3A_j$ where (we set $x_j=|(\nu_n)_j|$)
\begin{equation}
A_1(n)=O(n^{1+b})
\label{aone}
\end{equation}
\begin{equation}
A_2(n)=\sum_{j=\gamma_n+4}^{n-4}\frac{C^2}{j^{2b}}
\bigg|\left(1+O\left(j^{-2b}\right)\right)\left(1+O\left(j^{-1}\right)\right)x_{j-1}+
\left(1+O\left(j^{-2b}\right)\right)\left(1+O\left(j^{-1}\right)\right)x_{j+1}\bigg|^2
\label{atwo}
\end{equation}
\begin{equation}
A_3(n)=\sum_{{j=\gamma_n+4}\atop {j\, odd}}^{n-4}
\left|x_{j+2}-x_{j}+O\left(j^{-2b}\right)(x_{j}+x_{j+2})\right|^2
+\sum_{{j=\gamma_n+4}\atop {j\, even}}^{n-4}\left|x_{j-2}-x_{j}+O\left(j^{-2b}\right)(x_{j}+x_{j-2})\right|^2.
\label{athree}
\end{equation}
Additionally, using well-known formulas for $\sum_{j=1}^{Q}j^p$ for $p\in\{0,1,2,3,4\}$ (see formula 23.1.4 on page 804 in \cite{Handbook}) we can calculate
\begin{eqnarray}
\label{nunorm} \|\nu_{n}\|^2=\sum_{j=\gamma_{n}}^{n}(j-\gamma_{n})^2(n-j)^2=\frac{n^{5(1+b)/2}}{30}+o(n^{5(1+b)/2}).
\end{eqnarray}
This is the core of the necessary calculation.

\vspace{2mm}

\noindent\underline{\textit{Proof of Theorem \ref{kthzero}.}}
For each $n>0$, let us define $\phi_n=2\arcsin|f(n)|$.
To obtain a lower bound on the distance, we notice that if $\theta\in(-\phi_{M},\phi_{M})$ then by Lemma \ref{openchoice} the Blaschke product
\[
\frac{\eitheta\Phi_{M-1}(\eitheta)}{\Phi_{M-1}^*(\eitheta)}
\]
has argument in $(-\frac{\pi}{2}-\frac{1}{2}\phi_{M},\frac{\pi}{2}+\frac{1}{2}\phi_{M})$.  However, $\Phi_M^{(\beta)}$ has a zero precisely when this Blaschke product is equal to $\bar{\beta}$, which lies outside this arc for large $M$.  Finally, we notice that since $|f(n)|$ decays monotonically to $0$, we have
\[
|e^{i\phi_{M}}-1|\geq(2-\delta)|f(M)|
\]
for large $M$, which gives us the desired lower bound.

For the upper bound, we first consider the case $k=1$ and $\beta_n\equiv-1$.  By our earlier discussion, it suffices to show that if $M$ is such that the Verblunsky coefficients $\{\alpha_0,\ldots,\alpha_M,\ldots\}$ satisfy the conditions of the theorem then there exists a unit vector $v=v(M)$ such that
\[
||(\mcc^{(M)}_{-1}-1)v||<(2+\delta)|f(M)|.
\]
First let us consider the case $f(n)=-Cn^{-b}$ with $0<b<1$ and $C>0$.  We will show that we can set $v(M)=\nu_{M}/\|\nu_M\|$ as defined earlier with $\gamma_{M}=M-M^{\frac{1+b}{2}}$.  We begin by again using elementary formulas for $\sum_{j=1}^Qj^p$ to make the following calculations
\begin{eqnarray}
\label{sumsum} \sum_{j=\gamma_n+4}^{n-4}(x_{j+1}+x_{j-1})^2&=&\frac{4n^{5(1+b)/2}}{30}+o(n^{5(1+b)/2})\\
\label{diffsum} \sum_{j=\gamma_n+4}^{n-4}(x_{j}-x_{j\pm2})^2&=&O(n^{3(1+b)/2}).
\end{eqnarray}
We see from equations (\ref{aone}) and (\ref{nunorm}) that $A_1(n)/\|\nu_{n}\|^2=o(n^{-2b})$.  We can bound $A_2(n)$ (from equation (\ref{atwo})) by
\begin{eqnarray*}
A_2(n)&\leq&(1+o(1))C^2\sum_{j=\gamma_n+4}^{n-4}\frac{1}{j^{2b}}|x_{j+1}+x_{j-1}|^2\leq
\frac{(1+o(1))C^2}{n^{2b}}\sum_{j=\gamma_n+4}^{n-4}|x_{j+1}+x_{j-1}|^2\\
&=&\frac{4C^2}{30}n^{5(1+b)/2-2b}+o(n^{5(1+b)/2-2b})
\end{eqnarray*}
where we evaluated the last sum using equation (\ref{sumsum}).  Finally, we can bound $A_3(n)$ (see equation (\ref{athree})) from above by
\[
2\bigg(\sum_{{j=\gamma_n+4}\atop {j\, odd}}^{n-4}|x_{j}-x_{j+2}|^2+
\sum_{{j=\gamma_n+4}\atop {j\, even}}^{n-4}|x_{j}-x_{j-2}|^2\bigg)+
\frac{K}{n^{4b}}\bigg[\sum_{{j=\gamma_n+4}\atop {j\, odd}}^{n-4}|x_{j}+x_{j+2}|^2+
\sum_{{j=\gamma_n+4}\atop {j\, even}}^{n-4}|x_{j}+x_{j-2}|^2\bigg]
\]
for some constant $K>0$.  We can bound this further by eliminating the odd/even subscripts and then evaluate as before using equations (\ref{sumsum}) and (\ref{diffsum}) to see that
\[
A_3(n)=O(n^{5(1+b)/2-4b}+n^{3(1+b)/2}).
\]
Putting it all together, we conclude that
\[
\left(\frac{\|(\mccn_{-1}-1)\nu_{n}\|}{\|\nu_{n}\|}\right)^2\leq
\frac{4C^2}{n^{2b}}+o\left(\frac{1}{n^{2b}}\right)
\]
as $\nri$.  Therefore, if $M$ is large enough we can set $v(M)=\nu_M/\|\nu_M\|$ and get the desired conclusion.

If $f(n)$ is such that $\lim_{\nri}\sqrt{n}\,f^2(n-\sqrt{n})=\infty$ then we choose a different trial vector.
In this case we set
\[
(\upsilon_n)_j=
\begin{cases}
\sqrt{\frac{1}{\sqrt{n}}}, &\mbox{ $j\in(n-\sqrt{n},n)$}\\
0, &\mbox{ otherwise}
\end{cases}
\]
so that $\upsilon_n$ is a unit vector.  If $\alpha_n=f(n)+o(f(n))$ with $\sqrt{n}\,f^2(n-\sqrt{n})\rightarrow\infty$ then
\begin{eqnarray*}
\|(\mccn_{-1}-1)\upsilon_n\|^2&=&
\frac{1}{\sqrt{n}}\left(4\sum_{j=n-\sqrt{n}+1}^{n-1}\left(f^2(j)+o(f^2(j))\right)+O(1)\right)\\
&\leq& \left(4f^2(n-\sqrt{n})+O(n^{-1/2})\right)(1+o(1))=4f^2(n)(1+o(1))
\end{eqnarray*}
as $\nri$.  Therefore, if $M$ is large enough we can set $v(M)=\upsilon_M$ and get the desired conclusion.

This completes the proof of the $k=1$ case when $\beta_n\equiv-1$ (the statement concerning $\zeta_{M}^{(M)}$ follows by an obvious symmetry).  The conclusion for all $\beta$ in the desired range will follow from considering $k>1$ and using the fact that zeros of $\Phi_M^{(\beta)}(z)$ interlace for distinct values of $\beta$ (by Theorem 1.3 in \cite{RankOne}).

For $k>1$ we provide the details for the case $f(n)=-Cn^{-b}$ with $b<1$ and $C>0$ since the other case is a nearly identical calculation.  If we let $u_n=\nu_n/\|\nu_n\|$ then the calculation above actually shows that if a natural number $p>1$ is fixed then
\[
\|(\mccm_{(-1)}-1)u_{M-p(M^{\frac{1+b}{2}}+4)}\|\leq(2+\delta)CM^{-b}
\]
when $M$ is sufficiently large (we extend the vector $u_{M-p(M^{\frac{1+b}{2}}+4)}$ with zeros to make it a vector in $\bbC^{M}$).
Suppose we have an orthogonal collection $\{v_q\}_{q=1}^{m(M)}$ of unit vectors satisfying
\[
\|(\mccm_{(-1)}-1)v_p\|\leq(2+\delta/2)CM^{-b}
\]
and $m(M)\rightarrow\infty$ as $M\rightarrow\infty$.
Let $\lambda_M$ be the eigenvalue of $\mccm_{(-1)}$ closest to $1$ and let $\omega_M$ be the corresponding eigenvector. It follows that $\bar{\lambda}_M$ and $\bar{\omega}_M$ also form an eigenvalue-eigenvector pair.  Notice that since $\mccm_{(-1)}$ is unitary, it is a map from $\langle\omega_M,\bar{\omega}_M\rangle^\perp$ to itself.  Let $\tilde{v}_q$ be the projection of $v_q$ onto $\langle\omega_M,\bar{\omega}_M\rangle$ and let $w_q=v_q-\tilde{v}_q$.  Suppose
\[
\tilde{v}_q=a_q\omega_M+b_q\bar{\omega}_M.
\]
We calculate
\begin{eqnarray*}
\frac{\|(\mccm_{(-1)}-1)w_q\|}{\|w_q\|}&=&
\frac{\|(\mccm_{(-1)}-1)v_q-(\mccm_{(-1)}-1)\tilde{v}_q\|}{\sqrt{1-|a_q|^2-|b_q|^2}}
\leq\frac{(2+\delta/2)CM^{-b}+(|a_q|+|b_q|)|\lambda_M-1|}{\sqrt{1-|a_q|^2-|b_q|^2}}\\
&\leq&2CM^{-b}\left(\frac{(1+|a_q|+|b_q|)(1+\delta/4)}{\sqrt{1-|a_q|^2-|b_q|^2}}\right).
\end{eqnarray*}
Notice that
$1=\|\omega_M\|^2\geq\sum_{t=1}^{m(M)}|\langle\omega_M,v_t\rangle|^2$
so for any fixed $\epsilon>0$, the set
\[
X_M(\epsilon)=\{q:|a_q|^2=|\langle\omega_M,v_q\rangle|^2\leq \epsilon\}
\]
has cardinality tending to infinity as $M\rightarrow\infty$ (since $m(M)\rightarrow\infty$ as $M\rightarrow\infty$).
By similar reasoning, we have $1=\|\bar{\omega}_M\|^2\geq\sum_{t\in X_M(\epsilon)}|\langle\bar{\omega}_M,v_t\rangle|^2$ and so the set
\[
Y_M(\epsilon)=\{q\in X_M(\epsilon):|b_q|^2=|\langle\bar{\omega}_M,v_q\rangle|^2\leq \epsilon\}
\]
also has cardinality tending to infinity as $M\rightarrow\infty$.  Since $\epsilon>0$ can be chosen arbitrarily, we see that
\[
\min_{1\leq p\leq m(M)}\left\{\frac{(1+|a_p|+|b_p|)(1+\delta/4)}{\sqrt{1-|a_p|^2-|b_p|^2}}\right\}=1+\frac{\delta}{4}+o(1)
\]
so we have demonstrated the existence of an approximate eigenvector for $\mcc^{(M)}_{(-1)}$ in $\langle\omega_M,\bar{\omega}_M\rangle^\perp$.  It follows that $\mcc^{(M)}_{(-1)}$ must therefore have an eigenvector in $\langle\omega_M,\bar{\omega}_M\rangle^\perp$ with eigenvalue in the desired range.

We can repeat this procedure of projecting the vectors $\{v_p\}_{p=1}^{m(M)}$ to the span of the known eigenvectors.  By minimizing an expression of the form
\[
\frac{\|(\mccm_{(-1)}-1)w_p\|}{\|w_p\|}=
2CM^{-b}\left(\frac{(1+\delta/4)(1+\sum_{i}(|a_{p,i}|+|b_{p,i}|))}{\sqrt{1-\sum_{i}(|a_{p,i}|^2+|b_{p,i}|^2)}}\right)
\]
over all $p\leq m(M)$ we get the desired conclusion for arbitrary $k>1$ and $\beta=-1$.  The conclusion for all $\beta$ in the desired range follows by the interlacing of zeros.
\begin{flushright}
$\Box$
\end{flushright}

If we combine our result with some previously established results, we can understand some prominent features of Figure \ref{zeros}.
The picture suggests that for any fixed $k\in\bbN$,
\[
\left|\arg\left(\zeta_{k+1}^{(n)}\right)-\arg\left(\zeta_{k}^{(n)}\right)\right|=
\frac{2\pi}{n}+o\left(\frac{1}{n}\right)
\]
as $\nri$ (as was conjectured in \cite{FineI} when $f(n)=-Cn^{-b}$).  This was proven in \cite{FineIV} if we order the $\{\zeta_k^{(n)}\}_{k=1}^{n}$ counterclockwise starting from any $\eitheta\neq1$, that is, we have uniform clock spacing of the zeros away from $1$.  Our result proves only that the zero spacings closest to $1$ (but outside the large gap) are $o(n^{-b})$ as $\nri$.

\section{Extensions and Generalizations}\label{extensions}

Now we will examine ways in which we can tweak the hypotheses of Theorem \ref{kthzero} and obtain similar conclusions.  Throughout this section, we will let $k\in\bbN$, $\beta\in\partial\bbD$, and $\delta>0$ be fixed as in the statement of Theorem \ref{kthzero}.

For our first extension, we will relax the condition that $\alpha_n\in\bbR$ for all $n$.  Notice that a priori the proof of Theorem \ref{kthzero} only proves that the $k^{th}$ closest zero of $\Phi_M^{(\beta)}(z)$ to $1$ has argument approximately $\pm2|f(M)|$.  However, the reality assumption on the sequence $\{\alpha_n\}_{n\geq0}$ tells us that the zeros of $\Phi_M^{(-1)}(z)$ come in conjugate pairs so we can make a statement about $\zeta_k^{(M)}$ in particular.  We can in fact make the same conclusion by imposing a reality condition only on large blocks of the Verblunsky coefficients.  The size of these blocks will depend on the decay of the Verblunsky coefficients, as can be seen from the proof of Theorem \ref{kthzero}.

\begin{theorem}\label{nonrealextend}
Let $\{\hat{\alpha}_n\}_{n\geq0}$ be a sequence of real Verblunsky coefficients having slow decay controlled by $f$.  If $M$ and $N$ are sufficiently large and $\{\alpha_n\}_{n\geq0}$ is a sequence of Verblunsky coefficients satisfying $\alpha_n\in P_{\alpha_M}$ for all $n\leq M$ and $\alpha_n=\hat{\alpha}_n$ for all $n\in(M-N(t_M+4),M]$, then
\[
\left|\frac{\arg\left(\zeta_k^{(M)}\right)}{2|f(M)|}-1\right|<\delta.
\]
Here $t_M=M^{\frac{1+b}{2}}$ if $f(n)=-Cn^{b}$ for $C>0$ and $b\in(0,1)$ and $t_M=\sqrt{M}$ otherwise.
\end{theorem}

\begin{proof}
We will provide the details for the case $f(n)=-Cn^{-b}$ with $b\in(0,1)$ and $C>0$.  The other case can be handled with obvious modifications.

The proof of the lower bound given in the proof of Theorem \ref{kthzero} applies here so we need only prove the upper bound.
Let $\{a_n\}_{n\geq0}$ be the sequence of Verblunsky coefficients given by
\[
a_j=\alpha_{j-2+M-(N-1)(M^{\frac{1+b}{2}}+4)}
\]
with corresponding orthonormal polynomials $p_j(z)$ and let $c_j(z)=p_j(z)/p_j^*(z)$.  Now we define $\eta_n(\theta)$ to be the phase of $\eitheta b_{n-1}(\eitheta)$ so that $\eta_n$ is strictly increasing (by equation 10.8 in \cite{FineIV}) and changes by $2\pi n$ as $\theta$ runs from $0$ to $2\pi$.  Similarly, we define $\tau_n(\theta)$ to be the phase of $\eitheta c_{n-1}(\eitheta)$ (note that $\tau_n(0)=0$ for all $n$).  It follows from the formulas in Chapter 10.12 in \cite{OPUC2} that
\begin{eqnarray}\label{pruferecur}
\eta_n(\theta)=\eta_{n-1}(\theta)+\theta-2\arg(1-\alpha_{n-2}e^{i\eta_{n-1}(\theta)})
\end{eqnarray}
and similarly for $\tau_n(\theta)$ (notice the similarity to Proposition 2.2 in \cite{KStoi}).

Notice that the condition $\alpha_n\in P_{\alpha_M}$ for all $n\leq M$ implies (by Lemma \ref{openchoice})
\begin{eqnarray*}
\eta_{M-N(M^{\frac{1+b}{2}}+4)}(\thalphm)&>&
\thalphm-\frac{\pi}{2}+\arcsin|\alpha_M|.%
\end{eqnarray*}
By using equation (\ref{kru}) and the reasoning of Section \ref{keylemma}, one finds that if $\arg(b_{j}(e^{i\thalphm}))<0$ then after at most $O(M^b)$ more iterations of the recursion using negative real Verblunsky coefficients we have $\arg(b_{k}(e^{i\thalphm}))>0$.
A similar statement holds for $-\thalphm$.  Therefore, we can say
\[
\eta_{M-(N-1)(M^{\frac{1+b}{2}}+4)}((1+\delta)\thalphm)>(1+\delta)\thalphm=\tau_1((1+\delta)\thalphm).
\]
Now, for $n>M-(N-1)(M^{\frac{1+b}{2}}+4)=:m$ we have
\[
\eta_n(\theta)-\tau_{n-m+1}(\theta)=\eta_{n-1}(\theta)-\tau_{n-m}(\theta)
-2\arg\left(\frac{1-\alpha_{n-2}e^{i\eta_{n-1}(\theta)}}{1-\alpha_{n-2}e^{i\tau_{n-m}(\theta)}}\right).
\]
If $\alpha\in\bbR$ and $|\alpha|<1/5$, then the function $f(x)=\arg(1-\alpha e^{ix})$ is Lipschitz with Lipschitz constant strictly smaller than $1/2$.  Therefore, whenever $\eta_{n-1}(\theta)-\tau_{n-m}(\theta)>0$ the above formula shows $\eta_n(\theta)-\tau_{n-m+1}(\theta)>0$ too.  Therefore, by induction we get
\[
\eta_M((1+\delta)\thalphm)>\tau_{(N-1)(M^{\frac{1+b}{2}}+4)+1}((1+\delta)\thalphm).
\]
We can apply the proof of Theorem \ref{kthzero} to show that
\[
\tau_{(N-1)(M^{\frac{1+b}{2}}+4)+1}((1+\delta)\thalphm)>(2k+1)\pi
\]
if $N$ and $M$ are sufficiently large, so the same must be true of $\eta_M((1+\delta)\thalphm)$.  Furthermore, the proof of Lemma \ref{openchoice} shows that $\eta_M(0)\in(\frac{-1}{2}(\pi-\thalphm),\frac{1}{2}(\pi-\thalphm))$.  Therefore, $zb_{M-1}(z)$ wraps the arc $\{e^{ix}:x\in(0,(1+\delta)\thalphm)\}$ around the circle at least $k$ times so there are at least $k$ zeros of $\Phi_{M}^{(-1)}(z)$ in the desired range.  The result for general $\beta$ again follows from the interlacing of zeros.
\end{proof}

\noindent\textit{Remark.}  As in Theorem \ref{kthzero}, under the hypotheses of Theorem \ref{nonrealextend} we can make a corresponding statement concerning $\zeta_{M-k+1}^{(M)}$.

\vspace{2mm}

We can also state one generalization by applying our techniques to the orthogonal polynomials.  Lemma \ref{openchoice} shows that the phase of a certain Blaschke product stays away from $\pm\pi$ for $\theta\in\gap$.  Evaluating the same Blaschke product at $z=r\eitheta$ for $\theta\in\gap$, one can show that the phase is even farther from $\pm\pi$.  Therefore, one can use the same argument to prove the following

\begin{theorem}\label{wedge}
Let $\alpha\in(-1/2,0)$ and suppose $\alpha_j\in(-1,\alpha)$ for $j=0,1,\ldots,n-2$ and $\alpha_{n-1}=\alpha$.  If $W_\alpha$ is the sector of the unit disk subtending the arc $\gap$ then $\Phi_n(z)$ has no zeros in $W_\alpha$.
\end{theorem}

\noindent\textit{Remark.}  Theorem \ref{wedge} is also related to Conjecture D in \cite{FineI} and the analysis in \cite{Beware}.

\section{Appendix}\label{appendix}

Here we present an alternate approach to finding the lower bound for the distance from the closest zero of $\Phi_n^{(-1)}(z)$ to $1$ under the additional hypothesis that $\{\alpha_n\}_{n\geq0}$ is a sequence of real Verblunsky coefficients increasing monotonically to $0$.  This is a weaker result than what we obtained earlier, but the proof is very different.

\begin{lemma}\label{dets}
Let $J$ be an $n\times n$ real matrix of the following form (where $(+)$ indicates a positive element and $(-)$ indicates a negative one)
\[J=
\begin{pmatrix}
  +&-&0&0&\cdots&0 \\
  -&-&+&0&\cdots&0 \\
  0&+&+&-&\cdots&0 \\
  \vdots&&\ddots&\ddots&\ddots&\vdots\\
  0&&\cdots&0&-&-
\end{pmatrix}
\]
that is, the signs alternate along the three main diagonals and the first element along the main diagonal has the opposite sign of the first element of the off-diagonals.  Then $J$ is invertible.
\end{lemma}

\begin{proof}
We proceed by induction to prove that if $n=2m$ or $2m+1$ and $J_{1,1}>0$ then if $m$ is even the determinant of $J$ has the same sign as its $(1,1)$ element, while if $m$ is odd then then the determinant of $J$ has opposite sign as its $(1,1)$ element.  In either case, this shows the determinant is non-zero so $J$ is invertible.

The claim is easily verified for $m=1,2$ so the base case is trivial.  If $S,T\subseteq\{1,2,3,\ldots,n\}$, then we denote by $J_{(S)}^{(T)}$ the minor of $J$ with the rows numbered by elements of $S$ removed and the columns numbered by elements of $T$ removed.  With this notation, we have
\[
\det(J)=(-)\det\left(J_{\{1,2\}}^{\{1,2\}}\right)+(-)\det\left(J_{\{1,2,3\}}^{\{1,2,3\}}\right).
\]
Applying the induction hypothesis to these two smaller matrices gives the desired result.
\end{proof}

\begin{lemma}\label{cdetbound}
If $\{\alpha_n\}_{n\geq0}$ is a sequence of real Verblunsky coefficients increasing monotonically to $0$ then $||(\mccn_{-1}-1)^{-1}||<\frac{1}{2|\alpha_{n-1}|}$.
\end{lemma}

\begin{proof}
We give the proof for the case $n$ is even, the other case being nearly identical.  Using the factorization (4.2.18) in \cite{OPUC1}, we have
\[
\mccn_{-1}=\mcln\mcmn_{-1}.
\]
Therefore,
\[
||(\mccn_{-1}-1)^{-1}||=||(\mcln\mcmn_{-1}-1)^{-1}||
=||(\mcmn_{-1}-\mcln)^{-1}\mcln||
=||(\mcmn_{-1}-\mcln)^{-1}||
\]
since when $n$ is even, $\mcln$ is unitary (if $n$ is odd then $\mcmn$ is unitary).  Since all $\alpha_n\in\bbR$, the matrix $\mcmn_{-1}-\mcln$ is self-adjoint so this last quantity is equal to the reciprocal of the absolute value of the smallest eigenvalue of $\mcmn_{-1}-\mcln$.  However, if $\lambda\in(-2|\alpha_{n-1}|,2|\alpha_{n-1}|)$ then $\mcmn_{-1}-\mcln-\lambda$ is of the form given in Lemma $\ref{dets}$ and so is invertible and the Lemma follows.
\end{proof}

Now that we have a resolvent bound, it is a simple matter to derive the lower bound for $|\zeta_1^{\upp}-1|$.


\vspace{7mm}




\begin{thebibliography}{14}

\bibitem{Handbook}  M. Abramowitz, I. A. Stegun, {\em Handbook of Mathematical Functions}, Dover Publications, New York, New York, 1972.

\bibitem{Bandtlow} O. Bandtlow, {\em Estimates for Norms of Resolvents and an Application to the Perturbation of Spectra}, Math. Nachr. 267 (2004), 3--11.

\bibitem{DaviesSimon} E. B. Davies, B. Simon, {\em Eigenvalue Estimates for Non-normal Matrices and the Zeros of Random Orthogonal Polynomials on the Unit Circle}, J. Approx. Theory 141 (2006), 189--213.

\bibitem{Hislop}  P. D. Hislop, I. M. Sigal, {\em Introduction to Spectral Theory with applications to Schrodinger operators}, Springer Applied Mathematical Sciences, New York, New York, 1996.

\bibitem{KStoi} R. Killip, M. Stoiciu, {\em  Eigenvalue Statistics for CMV Matrices: From Poisson to Clock via Random Matrix Ensembles}, Duke Math. J. 146 (2009), no. 3, 361--399.

\bibitem{FineIV} Y. Last, B. Simon, {\em Fine Structure of the Zeros of Orthogonal Polynomials, IV.  A priori bounds and clock behavior}, Comm. Pure Appl. Math. 61 (2008), 486--538.

\bibitem{Nevai} P. Nevai, {\em Orthogonal Polynomials, Measures and Recurrences on the Unit Circle}, Trans. Amer. Math. Soc. 300 (1987), 175--189.

\bibitem{Beware} E. B. Saff, N.S. Stylianopoulos, {\em Asymptotics for Polynomial Zeros: Beware of Predictions from Plots}, Computational Methods and Function Theory 8 (2008), 385--407.

\bibitem{OPUC1} B. Simon, {\em Orthogonal Polynomials on the Unit Circle, Part One: Classical Theory}, American Mathematical Society, Providence, RI, 2005.

\bibitem{OPUC2}  B. Simon, {\em Orthogonal Polynomials on the Unit Circle, Part Two: Spectral Theory}, American Mathematical Society, Providence, RI, 2005.

\bibitem{FineI} B. Simon, {\em Fine Structure of the Zeros of Orthogonal Polynomials, I.  A Tale of Two Pictures}, Electronic Transactions on Numerical Analysis 25 (2006), 328--368.

\bibitem{RankOne} B. Simon, {\em Rank One Perturbations and the Zeros of Paraorthogonal Polynomials on the Unit Circle}, J. Math. Anal. Appl. 329 (2007), 376--382.

\bibitem{Stoiciu} M. Stoiciu, {\em The Statistical Distribution of the Zeroes of Random Paraorthogonal Polynomials on the Unit Circle}, J. Approx. Theory 139 (2006), 29--64.

\bibitem{Lilian} M.-W. L. Wong, {\em First and Second Kind Paraorthogonal Polynomials and Their Zeros}, J. Approx. Theory 146 (2007), 282--293.


\end{thebibliography}
\end{document}